\documentclass[conference]{IEEEtran}
\IEEEoverridecommandlockouts

\usepackage{cite}
\usepackage{amsmath,amssymb,amsfonts}
\usepackage{graphicx}
\usepackage{textcomp}
\usepackage{xcolor}
\usepackage{booktabs}
\usepackage{pgfplots}
\pgfplotsset{compat=1.17}
\usepackage{url}
\usepackage{amsthm}
\newtheorem{proposition}{Proposition}


\newcommand{\CostBase}{13003.67}   
\newcommand{\CostOne}{11954.67}    
\newcommand{\CostTwo}{12390.17}    
\newcommand{\CostThree}{12500.25}  
\newcommand{\CostFour}{12603.28}   
\newcommand{\CostFive}{12669.17}   
\newcommand{\CostSix}{12706.75}    

\newcommand{\VpiOne}{1049.00}
\newcommand{\VpiTwo}{613.50}
\newcommand{\VpiThree}{503.42}
\newcommand{\VpiFour}{400.39}
\newcommand{\VpiFive}{334.50}
\newcommand{\VpiSix}{296.92}

\newcommand{\PctOne}{8.07\%}
\newcommand{\PctTwo}{4.72\%}
\newcommand{\PctThree}{3.87\%}
\newcommand{\PctFour}{3.08\%}
\newcommand{\PctFive}{2.57\%}
\newcommand{\PctSix}{2.28\%}

\def\BibTeX{{\rm B\kern-.05em{\sc i\kern-.025em b}\kern-.08em
    T\kern-.1667em\lower.7ex\hbox{E}\kern-.125emX}}

\begin{document}

\title{The Value of Perfect Endpoint Forecasts for\\
Offshore-Wind Thermal Firming}

\author{\IEEEauthorblockN{Arash Khojaste}
\IEEEauthorblockA{\textit{Dept. of Mechanical and Industrial Engineering} \\
\textit{University of Massachusetts Amherst}\\
Amherst, MA, USA \\
akhojaste@umass.edu}
\and
\IEEEauthorblockN{Geoffrey Pritchard}
\IEEEauthorblockA{\textit{Dept. of Statistics} \\
\textit{University of Auckland}\\
Auckland, New Zealand \\
g.pritchard@auckland.ac.nz}
\and
\IEEEauthorblockN{Golbon Zakeri}
\IEEEauthorblockA{\textit{Dept. of Mechanical and Industrial Engineering} \\
\textit{University of Massachusetts Amherst}\\
Amherst, MA, USA \\
gzakeri@umass.edu}
}

\maketitle

\begin{abstract}
Forecast value depends not only on accuracy but also on the information
structure available to the operating model. We study a diagnostic
current-endpoint-only forecast for offshore-wind thermal firming: at hour $t$,
the controller observes the current net-demand state $z_t$ and one perfect
future target $z_{t+h}$, but not the intermediate path or earlier endpoint
messages. Unlike rolling path forecasts, these endpoint information sets are
not nested in $h$. We embed the signal in a cyclostationary MDP using quantile
Fourier regression states estimated from ISO New England load and offshore wind
data, and solve annual state-action-frequency LPs for $h=1,\ldots,6$. A
one-hour endpoint forecast reduces annual firming cost by \PctOne{}, while a
six-hour endpoint reduces it by \PctSix{}. The decreasing profile shows that a
single farther endpoint is less actionable for a one-step ramping decision,
without implying that longer rolling forecasts are less valuable.
\end{abstract}


\section{Introduction}
Large-scale offshore wind reduces energy-sector emissions, but it also increases
the variability of the residual demand that must be met by controllable
resources. System operators therefore schedule and dispatch \emph{firming}
capacity to cover the shortfall between demand and renewable output. In this
setting, net-demand forecasts are valuable because they allow the operator to
position thermal resources before high net-demand periods arrive. The central
question is not simply whether forecasts help, but which forecast information is
actually useful for the ramping decision made by the firming model.

A common intuition is that a longer perfect forecast should be at least as useful
as a shorter perfect forecast. That intuition is correct when the longer
forecast contains the shorter one. For example, a rolling perfect three-hour path
forecast at hour $t$ would reveal
$(z_{t+1},z_{t+2},z_{t+3})$, and therefore contain the one-hour forecast
$z_{t+1}$. In such a model, the information sets are nested, so the value of
perfect information cannot decrease with the forecast horizon.

This paper studies a different object: the value of a \emph{single perfect
endpoint}. At hour $t$, the operator observes $z_t$ and one endpoint
$z_{t+h}$. For $h=3$, the operator knows the current state and the state three
hours ahead, but not the one- and two-hour-ahead states. Thus, the endpoint
information sets for different $h$ are not nested: knowing $z_{t+3}$ does not
reveal $z_{t+1}$. This distinction is important. The model should not be interpreted as a
rolling three-hour forecast archive, nor as an operational claim that past
forecast messages are unavailable in practice. Instead, it is a controlled
value-of-information experiment that reveals exactly one future target state
and excludes all other future information, thereby isolating how much that
single target state is worth to a one-step ramping model.

The endpoint formulation is useful because it separates two questions that are
often conflated. Forecasts are commonly evaluated by target horizon, while an
operating model values information only through the decisions it changes. By
adding one target state $z_{t+h}$ to the decision state, the endpoint model asks:
how valuable is perfect knowledge of this particular future hour? It also avoids
the exponential state growth of a full path-valued forecast. A rolling path
forecast with $M$ net-demand states would require the net-demand component
$(z_t,z_{t+1},\ldots,z_{t+h})$, which has $M^{h+1}$ possible values. The
endpoint model uses only $(z_t,z_{t+h})$, so this component has $M^2$ values for
every horizon tested.

This work is related to three strands of literature. First, renewable forecast
value has been studied in stochastic unit commitment, reserve scheduling, and
market-operation models, where forecast scenarios or probabilistic forecasts
enter larger scheduling problems \cite{Morales,Zhou,Hobbs}. Second, the expected
value of perfect information is a classical tool in stochastic programming
\cite{Birge,Fourcade}. Third, forecast-aware sequential decision models and
Markov decision processes provide a natural way to represent information in the
state of the controller \cite{Puterman,Powell}. Our contribution is narrower but
more diagnostic: we isolate a single endpoint signal inside a cyclostationary
firming MDP and quantify how its value changes when the endpoint target moves
away from the one-hour ramping decision.

We quantify this endpoint value inside the cyclostationary firming MDP of
\cite{Khojaste2025-jh}. In that model, net demand is discretized by
quantile Fourier regression (QFR), producing a finite-state net-demand process
whose state definitions and transition probabilities vary periodically over the
year. QFR-based MDP models have also been used in other energy-planning
settings, including electricity price-aware data-center cooling and hydropower
reservoir operation, where periodic quantile regimes are used to construct
tractable Markovian state representations~\cite{Khojaste2025-lo,Pearce2025-fm}.
The firming resource is a thermal stack with discrete output levels and a
one-level-per-hour ramp limit. The baseline decision state is the pair consisting
of firming level and current net-demand state. We add exactly one coordinate: a
perfect forecast of the net-demand state $h$ hours ahead. Costs, ramp limits,
and the net-demand data are held fixed across all runs, so the difference in
optimal objective values measures the value of the endpoint information.

The main contributions are as follows:
\begin{itemize}
\item We formulate a forecast-augmented cyclostationary MDP in which a perfect
endpoint forecast $z_{t+h}$ is included directly in the decision state.
\item We give a precise interpretation of the information structure and
separate endpoint forecast value from rolling path forecast value.
\item We compute the value of perfect endpoint information for $h=1,\ldots,6$ on
an ISO New England offshore-wind firming case study.
\item We show empirically that endpoint value is largest when the endpoint
matches the one-hour ramping step. In the case study, the saving decreases from
\PctOne{} at $h=1$ to \PctSix{} at $h=6$.
\end{itemize}

The rest of the paper is organized as follows. Section~\ref{sec:model} reviews
the baseline cyclostationary firming MDP. Section~\ref{sec:endpoint} introduces
the endpoint forecast state and transition kernel. Section~\ref{sec:value}
defines the value of perfect endpoint information and explains why different
endpoint horizons are not ordered by a general monotonicity theorem.
Section~\ref{sec:case} describes the ISO-NE case study. Section~\ref{sec:results}
reports numerical results, Section~\ref{sec:limitations} discusses scope and
limitations, and Section~\ref{sec:conclusion} concludes.


\section{Baseline Cyclostationary Firming MDP}
\label{sec:model}

\subsection{Net-demand states}
Let $y_t$ denote hourly net demand, defined as electric load minus offshore-wind
generation. Net demand has strong daily and seasonal structure, so we model it as
a cyclostationary process over an annual period of $N=8760$ hours. Following
\cite{Khojaste2025-jh}, the continuous net-demand series is discretized
using QFR. For a probability level $p\in(0,1)$, the time-dependent quantile is
represented as
\begin{equation}
q_p(t)=\sum_{r=1}^{d}\beta_{p,r} b_r(t),
\label{eq:qfr}
\end{equation}
where $b_r(t)$ are periodic Fourier basis functions and the coefficients are
estimated by quantile regression \cite{Koenker}. Using quantile levels
$p_1<\cdots<p_{M-1}$, the net-demand state is
\begin{equation}
z_t=i
\quad\Longleftrightarrow\quad
q_{p_{i-1}}(t)\le y_t<q_{p_i}(t),
\qquad i\in\mathcal{I},
\label{eq:state}
\end{equation}
with $q_{p_0}(t)=-\infty$ and $q_{p_M}(t)=+\infty$. In the case study we use
$M=4$ states, separated by the conditional 0.25, 0.50, and 0.75 quantiles. Thus,
state labels have time-varying meanings: state 4 represents high net demand for
that hour of the day and season, not a fixed MW interval.

The baseline net-demand transition matrix is time-inhomogeneous and periodic:
\begin{align}
p_{ij}(t)
&=\Pr(z_{t+1}=j\mid z_t=i) \notag\\
&=\alpha_{ij}+\beta^c_{ij}\cos\!\left(\frac{2\pi t}{N}\right) \notag\\
&\quad +\beta^s_{ij}\sin\!\left(\frac{2\pi t}{N}\right).
\label{eq:baseline_kernel}
\end{align}
The coefficients are estimated by maximum likelihood subject to
$p_{ij}(t)\ge0$ and $\sum_j p_{ij}(t)=1$ for every hour $t$.
Following \cite{Khojaste2025-jh}, the QFR discretization uses $r=2$
harmonics ($d=5$ basis terms in \eqref{eq:qfr}), and all periodic transition
kernels are estimated by maximum likelihood with row-wise nonnegativity and unit
row sums enforced through a constrained nonlinear program at every hour $t$.

\subsection{Firming levels, actions, and cost}
The firming resource is a thermal stack represented by $L_R$ discrete capacity
levels. Level $\ell$ provides firm capacity $Q(\ell)$, increasing in $\ell$. At
each hour the operator chooses
\begin{equation}
a_t\in\mathcal{A}=\{-1,0,1\},
\end{equation}
where $a_t=1$ ramps the firming level up by one step, $a_t=0$ holds it, and
$a_t=-1$ ramps it down by one step. Boundary actions that would move below
level 1 or above level $L_R$ are infeasible and are omitted from the balance
equations.

The hourly cost incurred when the system is at hour $t$, firming level $\ell$,
and realized net-demand state $z$ is
\begin{equation}
c(t,\ell,z)=c_{\mathrm r}Q(\ell)
+c_{\mathrm d}\bigl(D_t(z)-Q(\ell)\bigr)^+,
\label{eq:cost}
\end{equation}
where $D_t(z)$ is the representative net demand in state $z$ at hour $t$,
$c_{\mathrm r}$ is the per-unit cost of holding firm capacity, and
$c_{\mathrm d}$ is the penalty for unmet net demand. The action selected at hour
$t$ determines the firming level available at hour $t+1$. Therefore, the current
action is valuable only to the extent that it positions the resource for the
next realized net-demand state.

\subsection{State-action frequency LP}
Without forecasts, the decision state is $(\ell,z)$. The optimal periodic policy
can be computed from a state-action frequency linear program. Let
$x_{t,\ell,z,a}$ be the probability of being in state $(\ell,z)$ and taking
action $a$ at hour $t$ in the periodic regime. The baseline LP is
\begin{subequations}\label{eq:baseline_lp}
\begin{align}
\min_{x\ge0}\quad
&\sum_t\sum_{\ell,z,a}c(t,\ell,z)x_{t,\ell,z,a}
\label{eq:base_obj}\\
\text{s.t.}\quad
&\sum_{\ell,z,a}x_{t,\ell,z,a}=1,\qquad \forall t,
\label{eq:base_norm}\\
&\sum_a x_{t+1,\ell,z,a}
-\sum_{z'}p_{z'z}(t)\bigl(
 x_{t,\ell,z',0} \notag\\
&\qquad
+x_{t,\ell+1,z',-1}
+x_{t,\ell-1,z',1}
\bigr)=0,
\label{eq:base_balance}
\end{align}
\end{subequations}
for all valid $t$, $\ell$, and $z$, with invalid boundary predecessor terms
omitted. Hour $N+1$ is identified with hour 1. The optimal randomized policy is
obtained by normalizing $x^\star$ over actions.

\section{Endpoint-Forecast Augmentation}
\label{sec:endpoint}

\subsection{Information structure}
For a fixed endpoint horizon $h$, define
\begin{equation}
e_t^{(h)}=z_{t+h}.
\end{equation}
The forecast-augmented decision state is
\begin{equation}
(\ell_t,z_t,e_t^{(h)}).
\label{eq:aug_state}
\end{equation}
This is the only structural change from the baseline model. The cost function,
action set, ramp dynamics, and net-demand data are unchanged.

For $h>1$, the controller observes neither the intermediate states
$(z_{t+1},\ldots,z_{t+h-1})$ nor previous endpoint messages. Thus, for
$h=3$, the model observes $(z_t,z_{t+3})$, not the full path
$(z_t,z_{t+1},z_{t+2},z_{t+3})$.

This distinction is the core interpretation of the paper. The endpoint
information sets
\begin{equation}
\sigma(z_t,z_{t+1}),\quad \sigma(z_t,z_{t+2}),\quad \ldots,\quad
\sigma(z_t,z_{t+h})
\end{equation}
are generally not nested. Consequently, a longer endpoint is not automatically
more valuable than a shorter endpoint. A rolling path forecast would have a
different state and a different monotonicity property; we return to this point in
Section~\ref{sec:value}.

\subsection{Joint endpoint transition kernel}
For a fixed $h$, the augmented demand state is the pair
$(z_t,e_t^{(h)})=(z_t,z_{t+h})$. Its one-step transition is
\begin{equation}
(z_t,z_{t+h})\longrightarrow (z_{t+1},z_{t+1+h}).
\end{equation}
We estimate the time-dependent transition probability
\begin{align}
P_t\lbrack(i,i')\to(j,j')\rbrack
&=\Pr\{(z_{t+1},z_{t+1+h})=(j,j') \notag\\
&\qquad\mid (z_t,z_{t+h})=(i,i')\} \notag\\
&=\alpha_{ii'jj'}+\beta^c_{ii'jj'}\cos\!\left(\frac{2\pi t}{N}\right) \notag\\
&\quad+\beta^s_{ii'jj'}\sin\!\left(\frac{2\pi t}{N}\right).
\label{eq:endpoint_kernel}
\end{align}
The coefficients are fitted by maximum likelihood from the observed transitions
$(z_t,z_{t+h})\to(z_{t+1},z_{t+1+h})$, separately for each $h$. The estimation is
constrained so that each row of $P_t$ is a probability distribution for every
hour $t$:
\begin{equation}
P_t\lbrack(i,i')\to(j,j')\rbrack\ge0,
\qquad
\sum_{j,j'}P_t\lbrack(i,i')\to(j,j')\rbrack=1.
\end{equation}
For $h=1$, the perfect endpoint is the next net-demand state, so admissible
transitions have a deterministic support condition: the next current state must
equal the previous endpoint. In the notation above,
$P_t\lbrack(i,i')\to(j,j')\rbrack=0$ whenever $j\ne i'$. This restriction is
part of the perfect one-hour endpoint specification.

For a clean value-of-information comparison with the no-forecast baseline, row
stochasticity of $P_t$ is not by itself sufficient. The current-state marginal
induced by the endpoint process should also be consistent with the baseline
one-step chain. If $\rho_t^{(h)}(e\mid i)$ denotes the hour-$t$ conditional
distribution of the endpoint coordinate given $z_t=i$, this condition can be
written as
\begin{equation}
\sum_{e,e'}\rho_t^{(h)}(e\mid i)
P_t\lbrack(i,e)\to(j,e')\rbrack=p_{ij}(t),
\qquad \forall i,j,t.
\label{eq:marginal_consistency}
\end{equation}
The non-negativity result in Section~\ref{sec:value} is stated under this
marginal-consistency condition. Because each endpoint kernel is fitted
separately by horizon, a nonnegative $\mathrm{VPI}_{\mathrm{end}}(h)$ should be
interpreted as an objective comparison for that fitted model, and as a
consequence of Proposition~\ref{prop:nonnegative} only when
\eqref{eq:marginal_consistency} holds or has been verified.

\subsection{Forecast-augmented LP}
Let $x_{t,\ell,z,e,a}$ be the state-action frequency for hour $t$, firming level
$\ell$, current net-demand state $z$, endpoint forecast state $e$, and action
$a$. The endpoint-forecast LP is
\begin{subequations}\label{eq:endpoint_lp}
\begin{align}
\min_{x\ge0}\quad
&\sum_t\sum_{\ell,z,e,a}c(t,\ell,z)x_{t,\ell,z,e,a}
\label{eq:end_obj}\\
\text{s.t.}\quad
&\sum_{\ell,z,e,a}x_{t,\ell,z,e,a}=1,
\qquad \forall t,
\label{eq:end_norm}\\
&\sum_a x_{t+1,\ell,z,e,a}
-\sum_{z',e'}P_t\lbrack(z',e')\to(z,e)\rbrack
\bigl(
 x_{t,\ell,z',e',0} \notag\\
&\qquad
+x_{t,\ell+1,z',e',-1}
+x_{t,\ell-1,z',e',1}
\bigr)=0.
\label{eq:end_balance}
\end{align}
\end{subequations}
As in the baseline LP, invalid boundary terms are omitted and hour $N+1$ wraps
to hour 1. The forecast coordinate affects the objective only through the
policy: the cost in \eqref{eq:end_obj} depends on the realized current state
$z$, not directly on the forecast state $e$.

For each endpoint horizon $h$, solving \eqref{eq:endpoint_lp} yields an optimal
annual cost $C_h$. The no-forecast baseline cost is denoted by $C_0$.

\section{Value of Perfect Endpoint Information}
\label{sec:value}

\subsection{Definition}
We define the value of perfect endpoint information at horizon $h$ as
\begin{equation}
\mathrm{VPI}_{\mathrm{end}}(h)=C_0-C_h.
\label{eq:vpi_end}
\end{equation}
This is the reduction in optimal expected annual firming cost obtained by adding
the endpoint state $z_{t+h}$ to the controller's information.

\subsection{Forecast information cannot hurt the baseline}
\begin{proposition}\label{prop:nonnegative}
Suppose the endpoint kernel has the same one-step net-demand marginal as the
baseline model. Then $C_h\le C_0$ and
$\mathrm{VPI}_{\mathrm{end}}(h)\ge0$ for every fixed endpoint horizon $h$.
\end{proposition}

\begin{IEEEproof}
Take any feasible no-forecast policy $d_{t,\ell,z,a}$. In the endpoint-forecast
model, define a policy that ignores the endpoint coordinate:
\begin{equation}
\tilde d_{t,\ell,z,e,a}=d_{t,\ell,z,a}\qquad \forall e.
\end{equation}
Under the stated marginal-consistency condition, this lifted policy induces the
same distribution over $(\ell,z)$ and the same expected cost as the original
policy. Therefore the forecast-augmented feasible policy class contains a policy
with cost $C_0$. Since $C_h$ is the minimum over the augmented class,
$C_h\le C_0$.
\end{IEEEproof}

Proposition~\ref{prop:nonnegative} is not a monotonicity theorem across
endpoint horizons. It only says that, for a fixed $h$, adding the endpoint
coordinate cannot make the optimum worse than a model that ignores it.

\subsection{Why endpoint horizons are not ordered}
For rolling path forecasts, the information sets are nested:
\begin{equation}
\sigma(z_t,z_{t+1})\subseteq
\sigma(z_t,z_{t+1},z_{t+2})\subseteq\cdots.
\end{equation}
Thus, under a full path forecast, the value of perfect information would be
weakly nondecreasing with horizon.

For endpoint forecasts, the information sets are instead
\begin{equation}
\sigma(z_t,z_{t+1}),\quad
\sigma(z_t,z_{t+2}),\quad
\sigma(z_t,z_{t+3}),\quad\ldots,
\end{equation}
which are generally non-nested. A three-hour endpoint forecast does not reveal
the one-hour endpoint. Therefore the classical argument that ``more perfect
lookahead is always better'' does not apply.

This is exactly why endpoint value can decline with $h$. The action chosen at
hour $t$ changes the firming level available at hour $t+1$. Hence the most
actionable future state for the current decision is $z_{t+1}$. When $h=1$, the
endpoint forecast reveals that state exactly. When $h=6$, the endpoint forecast
reveals $z_{t+6}$, which may be correlated with $z_{t+1}$ but does not determine
it. The farther endpoint can still be useful, but it is less directly aligned
with the one-step ramping decision.

\section{ISO New England Case Study}
\label{sec:case}

We use the ISO New England offshore-wind firming setting of
\cite{Khojaste2025-jh}. Hourly regional load for 2006--2020 is taken
from FERC Form 714 \cite{FERC}. Offshore-wind generation is computed from hourly
wind speeds at NOAA buoy 44025 \cite{NOAA}, converted to generation using the
IEA 15 MW reference turbine power curve \cite{IEA15}, and scaled to an offshore
wind capacity of 21,687 MW. Net demand is load minus offshore-wind generation.

The QFR state model uses $M=4$ net-demand states defined by the 0.25, 0.50, and
0.75 conditional quantiles. The firming stack has $L_R=14$ capacity levels and
action set $\{-1,0,1\}$. The cost parameters are shown in Table~\ref{tab:params}.
The coefficients $c_{\mathrm r}$ and $c_{\mathrm d}$ are normalized objective
weights rather than dollar-denominated prices. Accordingly, the absolute costs
reported below are in the normalized units of \eqref{eq:cost}; percentage
savings and cost differences under identical weights are the primary economic
comparisons.

\begin{table}[t]
\centering
\caption{Firming MDP parameters used in the case study.}
\label{tab:params}
\begin{tabular}{lll}
\toprule
Symbol & Meaning & Value \\
\midrule
$N$ & Hours in annual cycle & 8760 \\
$M$ & Net-demand states & 4 \\
$L_R$ & Firming levels & 14 \\
$\mathcal{A}$ & Ramp actions & $\{-1,0,1\}$ \\
$c_{\mathrm r}$ & Firm-capacity cost coefficient & 0.1 \\
$c_{\mathrm d}$ & Unserved-demand penalty coefficient & 3 \\
$h$ & Endpoint horizons tested & $1,\ldots,6$ hours \\
\bottomrule
\end{tabular}
\end{table}

For each endpoint horizon $h$, we estimate the joint endpoint kernel
\eqref{eq:endpoint_kernel} from the paired state sequence
$(z_t,z_{t+h})\to(z_{t+1},z_{t+1+h})$ and solve \eqref{eq:endpoint_lp}. Each
forecast-augmented LP has
$N L_R M^2 |\mathcal{A}|\approx 5.9\times10^6$ variables, while the no-forecast
baseline has $N L_R M |\mathcal{A}|$ variables. The transition kernels are sparse
and reused across firming levels. The instances are solved with Gurobi's barrier
method, and all reported costs are optimal objective values certified by the
solver within its feasibility and optimality tolerances. As a consistency check, we compared the current-state marginals implied by each
fitted endpoint kernel with the baseline one-step QFR transition matrix on the
same state sequence. Across $h=1,\ldots,6$, the maximum absolute discrepancy was
$0.0199$ and the mean absolute discrepancy was $0.0026$.

\section{Results}
\label{sec:results}

Table~\ref{tab:vpi} gives the optimal annual firming cost and value of perfect
endpoint information for each horizon. Figure~\ref{fig:vpi} shows the same
results graphically.

\begin{table}[t]
\centering
\caption{Annual firming cost and value of perfect endpoint information.}
\label{tab:vpi}
\begin{tabular}{cccc}
\toprule
Endpoint $h$ & Cost $C_h$ & $\mathrm{VPI}_{\mathrm{end}}(h)$ & Saving \\
(hours) & & $=C_0-C_h$ & (\%) \\
\midrule
0 & \CostBase  & --         & -- \\
1 & \CostOne   & \VpiOne    & \PctOne \\
2 & \CostTwo   & \VpiTwo    & \PctTwo \\
3 & \CostThree & \VpiThree  & \PctThree \\
4 & \CostFour  & \VpiFour   & \PctFour \\
5 & \CostFive  & \VpiFive   & \PctFive \\
6 & \CostSix   & \VpiSix    & \PctSix \\
\bottomrule
\end{tabular}
\end{table}

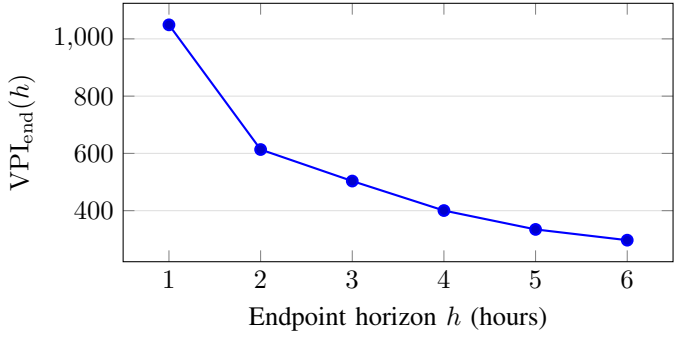
\begin{figure}[t]
\centering
\begin{tikzpicture}
\begin{axis}[
    width=\columnwidth, height=5.0cm,
    xlabel={Endpoint horizon $h$ (hours)},
    ylabel={$\mathrm{VPI}_{\mathrm{end}}(h)$},
    xmin=0.5, xmax=6.5, xtick={1,2,3,4,5,6},
    ymajorgrids, grid style={gray!25}, mark size=2pt,
]
\addplot+[thick, mark=*] coordinates {
    (1, 1049.00) (2, 613.50) (3, 503.42) (4, 400.39) (5, 334.50) (6, 296.92)
};
\end{axis}
\end{tikzpicture}
\caption{Value of perfect endpoint information as a function of endpoint
horizon. The value is largest when the endpoint is the next dispatch hour.}
\label{fig:vpi}
\end{figure}


\subsection{Endpoint information has positive value}
All endpoint forecast runs reduce cost relative to the no-forecast baseline:
\begin{equation}
C_h<C_0,\qquad h=1,\ldots,6.
\end{equation}
This is consistent with Proposition~\ref{prop:nonnegative}. In the numerical
solutions, every $\mathrm{VPI}_{\mathrm{end}}(h)$ in Table~\ref{tab:vpi} is
strictly positive. Even the six-hour endpoint forecast, which is not directly
aligned with the next ramping step, reduces annual cost by \VpiSix{}, or
\PctSix{} of the baseline objective.

\subsection{The nearest endpoint is most valuable}
The endpoint values are strictly ordered in the case study:
\begin{equation}
\mathrm{VPI}_{\mathrm{end}}(1)>
\mathrm{VPI}_{\mathrm{end}}(2)>
\cdots>
\mathrm{VPI}_{\mathrm{end}}(6)>0.
\end{equation}
Although $h=1$ is a boundary case of the endpoint construction because
$e_t=z_{t+1}$ and admissible transitions satisfy the deterministic support
condition $j=i'$, the decreasing pattern is not driven solely by this special
case. Over $h=2,\ldots,6$, where the endpoint no longer determines the next
current state, $\mathrm{VPI}_{\mathrm{end}}(h)$ still decreases from
\VpiTwo{} to \VpiSix{}.
A one-hour endpoint forecast reduces annual firming cost by \VpiOne{}, or
\PctOne{}. A six-hour endpoint forecast reduces cost by \VpiSix{}, or \PctSix{}.
Thus the six-hour endpoint retains only about 28\% of the value of the one-hour
endpoint.

This pattern does not contradict the value-of-information principle,
because the endpoint information sets are not nested. It reflects the timing of
the firming decision. The current action controls the firming level one hour in
the future, so the most useful endpoint is $z_{t+1}$. Farther endpoints still
carry statistical information about the evolution of net demand, but they do not
reveal the state that is immediately affected by the current ramp action.

\subsection{Policy interpretation}
The cost reductions in Table~\ref{tab:vpi} are reflected in the optimal ramping
rules. Figure~\ref{fig:policy_slices} compares the no-forecast baseline policy
with two slices of the one-hour endpoint policy on a representative summer day.
The no-forecast policy conditions on $(\ell_t,z_t)$ only. By contrast, the
one-hour endpoint policy conditions on $(\ell_t,z_t,e_t)$, where
$e_t=z_{t+1}$ is the forecasted next net-demand state. Thus the same current
state $(\ell_t,z_t)$ may lead to different actions depending on the predicted
next state.

\begin{figure*}[!t]
\centering
\includegraphics[width=0.56\textwidth]{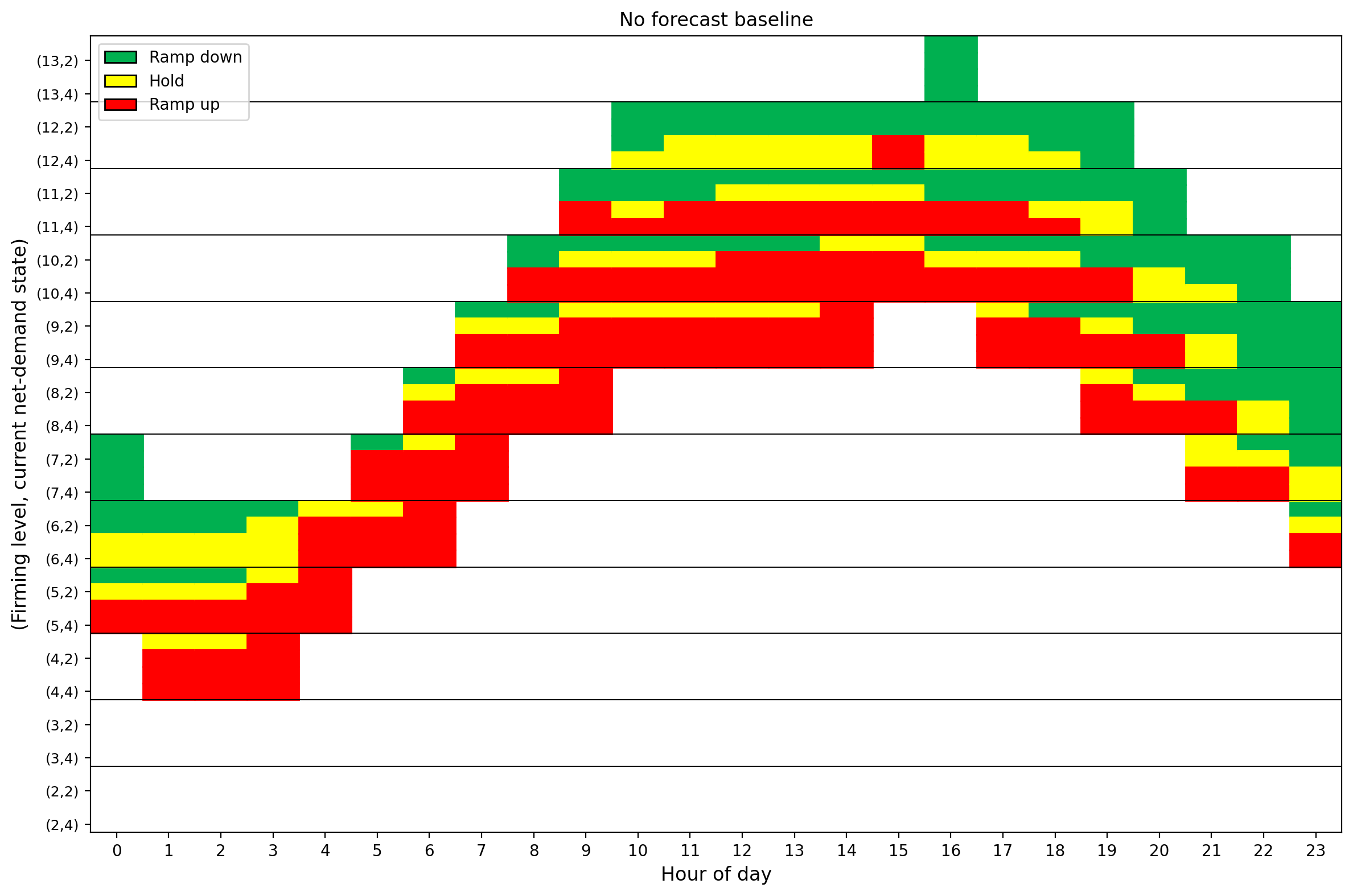}

\vspace{0.5em}

\begin{minipage}[t]{0.49\textwidth}
\centering
\includegraphics[width=\linewidth]{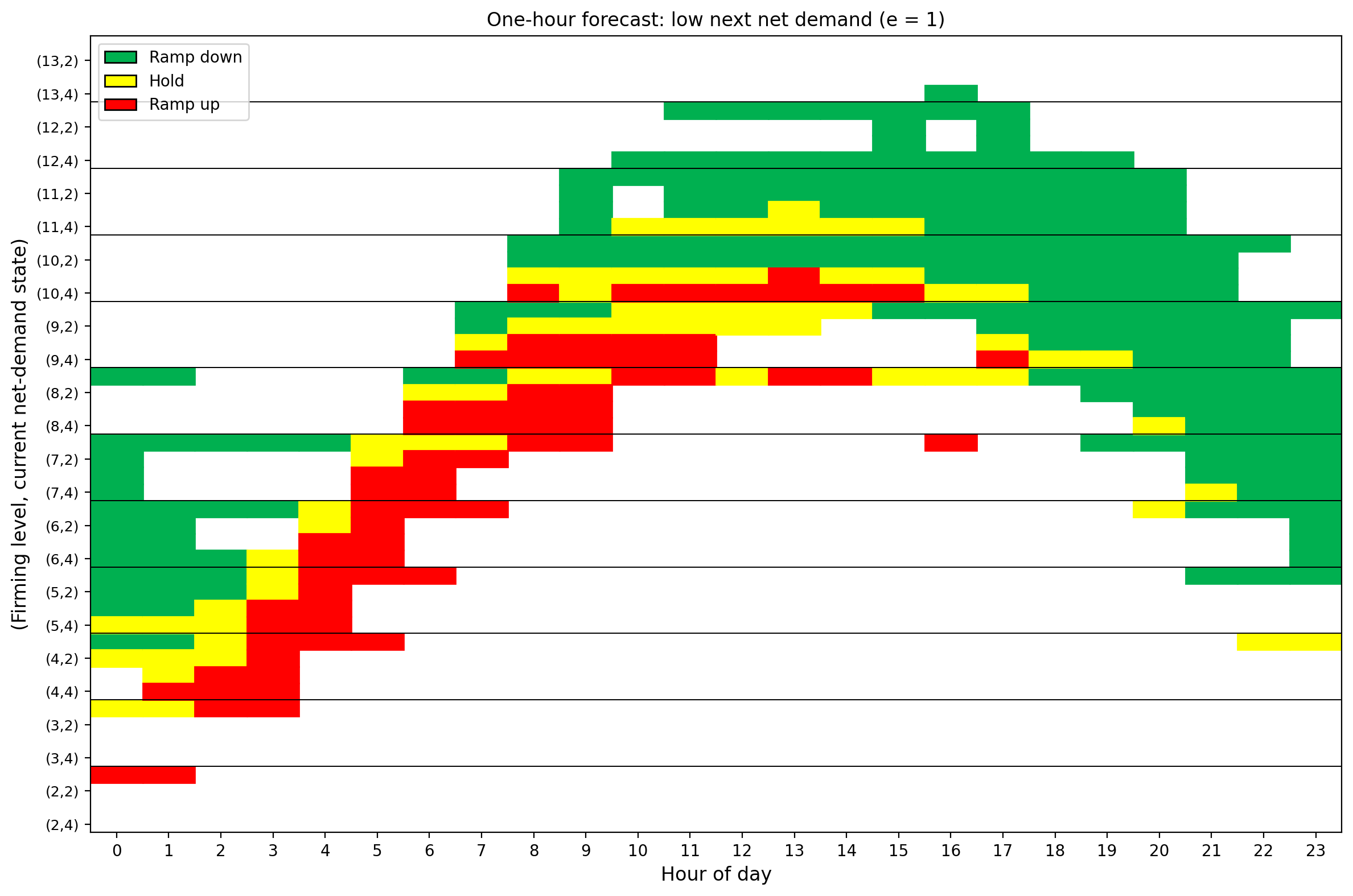}
\end{minipage}\hfill
\begin{minipage}[t]{0.49\textwidth}
\centering
\includegraphics[width=\linewidth]{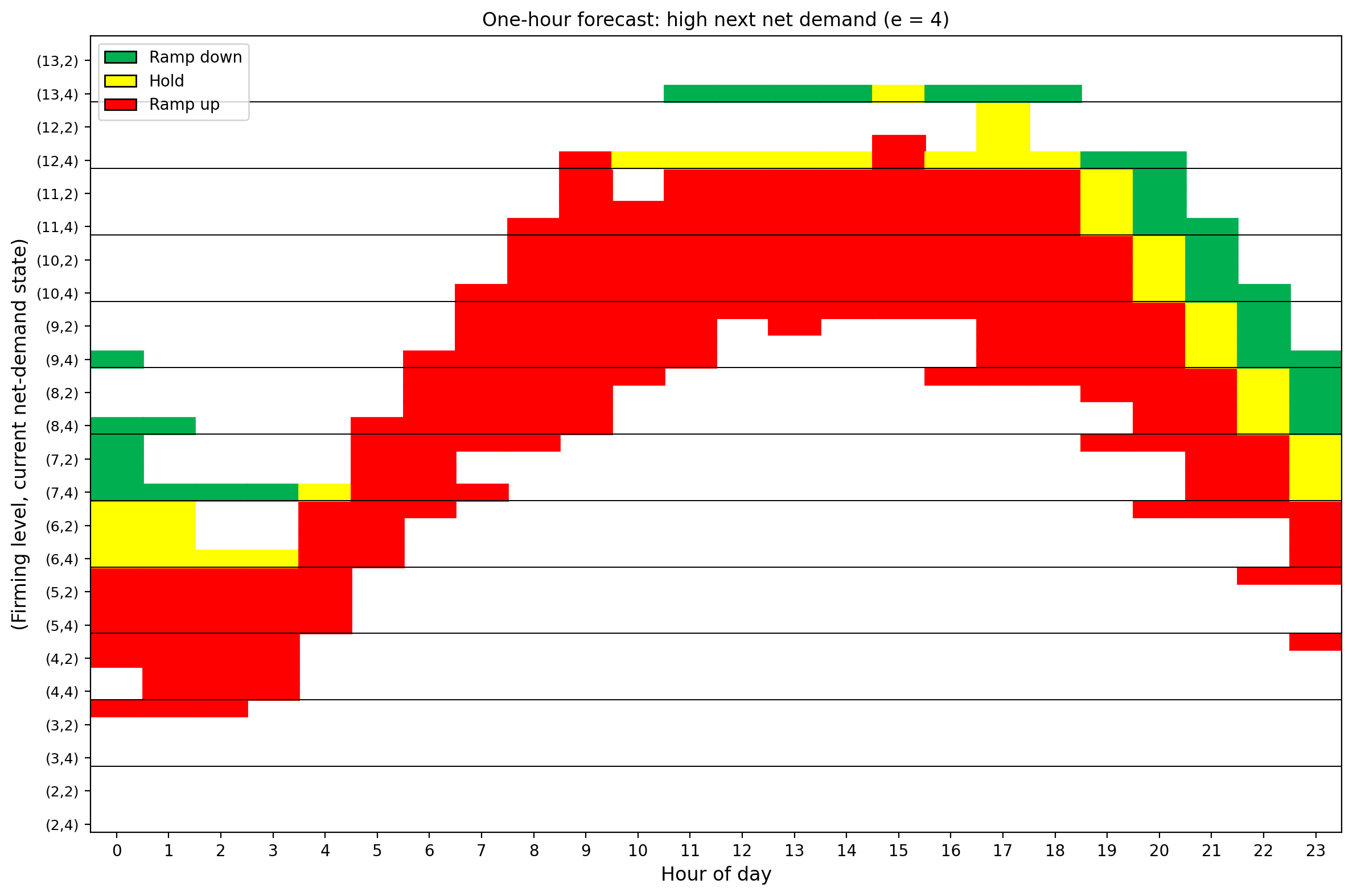}
\end{minipage}
\caption{Optimal ramping policies on July 15. The top panel shows the
no-forecast baseline, which conditions only on the firming level and current
net-demand state $(\ell_t,z_t)$. The bottom panels show two slices of the
one-hour perfect endpoint policy, which also conditions on the forecasted
next net-demand state $e_t=z_{t+1}$. The vertical axis lists firming-level/current-state
pairs that appear in the plotted support on this representative day. White cells are
state-hour pairs outside that support; colored cells show the selected ramp action
where the policy is defined. When the
next state is forecast to be low ($e=1$), the controller ramps down or holds
over a larger part of the state space. When the next state is forecast to be
high ($e=4$), the ramp-up region expands. Green, yellow, and red denote ramp
down, hold, and ramp up, respectively.}
\label{fig:policy_slices}
\end{figure*}

Figure~\ref{fig:policy_slices} illustrates why the one-hour endpoint has the
largest value. The action chosen at hour $t$ determines the firming level at
hour $t+1$. Therefore, direct information about $z_{t+1}$ changes the optimal
ramping rule in the current hour: a high predicted next state shifts the rule
toward ramping up, while a low predicted next state shifts it toward holding or
ramping down.

\subsection{Implications for forecast evaluation}
The result suggests a practical caution for planners. Forecast value should be
measured using the same information structure and timing as the decision model
that consumes the forecast. If the operational product is a rolling path forecast,
then the correct state should include the path or a sufficient statistic of that
path, and forecast value should be weakly increasing with horizon. If the product
or experiment supplies only one target hour, then the endpoint value may peak at
the target that is closest to the actuation delay of the resource.

For the thermal firming system studied here, the actuation delay is one hour and
one ramp level. Therefore, endpoint information is most valuable at $h=1$. A
resource with faster ramping, storage-like intertemporal flexibility, or a
multi-hour commitment delay could produce a different endpoint-value profile.
The method developed here can be reused for those technologies by changing the
state and transition dynamics.

\section{Scope and Limitations}
\label{sec:limitations}

The study has three main limitations. First, endpoint forecasts are perfect, so
the reported values are upper bounds for imperfect forecast products. Second,
the endpoint-only state is a diagnostic information structure rather than an
operational forecast archive; in practice, previously issued forecasts may be
retained, whereas here they are excluded to isolate the value of one future
target state. Third, the firming model uses one aggregate thermal stack with
one-level-per-hour ramping and omits unit commitment, startup costs, minimum
up/down constraints, network constraints, and co-optimized reserves. These
features could change which forecast horizons are most actionable.

\section{Conclusion}
\label{sec:conclusion}

This paper quantified the value of perfect endpoint forecasts for offshore-wind
thermal firming in a cyclostationary MDP. The forecast-augmented model adds one
coordinate to the baseline state: the net-demand state $h$ hours ahead. This
state is intentionally endpoint-only and does not represent a rolling path
forecast. Under this information structure, endpoint horizons are non-nested, so
there is no theorem requiring farther endpoints to be more valuable.

In the ISO New England case study, all endpoint forecasts reduce cost relative
to the no-forecast baseline, but the nearest endpoint is the most valuable. The
one-hour endpoint forecast reduces annual firming cost by \PctOne{}, while the
six-hour endpoint forecast reduces it by \PctSix{}. The result is explained by
the one-hour ramping dynamics: the current action positions the firming level for
hour $t+1$, so the endpoint $z_{t+1}$ is directly actionable, while a farther
endpoint only indirectly informs that decision.

The broader message is that forecast value must be matched to the timing and
information structure of the decision problem. Endpoint forecasts and rolling
path forecasts answer different economic questions. Future work will combine
endpoint forecasts with risk-averse MDP objectives such as CVaR, building on
recent risk-aware grid and reservoir models~\cite{Khojaste2026-jc}. We
will also study imperfect forecasts and compact path-valued approximations.

\section*{Acknowledgment}
This material is based upon work supported by the U.S. Department of Energy's
Office of Energy Efficiency and Renewable Energy (EERE) under the Wind Energy
Technologies Office (WETO) Award Number DE-EE0011269, the Massachusetts Clean
Energy Center and the Maryland Energy Administration. The views expressed herein
do not necessarily represent the views of the U.S. Department of Energy, the
United States Government, the Massachusetts Clean Energy Center or the Maryland
Energy Administration.

\bibliographystyle{IEEEtran}
\bibliography{references}

\end{document}